\documentclass[11pt]{amsart}

\usepackage{amsmath}
\usepackage{amssymb}
\usepackage{mathrsfs}
\usepackage[all]{xy}
\usepackage{xcolor}

\newtheorem{theorem}{Theorem}[section]

\newtheorem{lemma}[theorem]{Lemma}

\newtheorem{proposition}[theorem]{Proposition}
\newtheorem{corollary}[theorem]{Corollary}

\theoremstyle{definition}
\newtheorem{definition}[theorem]{Definition}
\newtheorem{example}[theorem]{Example}

\newtheorem{question}[theorem]{Question}

\theoremstyle{remark}
\newtheorem{remark}[theorem]{Remark}

\newtheorem{fact}[theorem]{Fact}

\def\l{{\langle}}
\def\r{{\rangle}}

\def\G{{\text{G}(\vec{\mathcal{A}})}}
\def\A{{\vec{\mathcal{A}}}}
\def\sky{{\text{sky}}}

\title[asympptotics of Guessing Sequences] {On the Asymptotic Behavior of Guessing Sequences}

\author{Tom Benhamou}
\address[Tom Benhamou]{Department of Mathematics, Rutgers University, New Brunswick, NJ USA}
\email{tom.benhamou@rutgers.edu}
\thanks{This research  was supported by the National Science Foundation under Grant
No. DMS-2346680}

\author{Sean LeClair}
\address[Sean LeClair]{Department of Mathematics, Rutgers University, New Brunswick, NJ USA}
\email{sean.leclair@rutgers.edu}
\thanks{This paper was the result of an undergraduate research project at Rutgers University New Brunswick}


\date{January 2025}

\begin{document}

\begin{abstract}
    We continue the study of probabilistic and topological properties of the set of reals that are being guessed by a diamond sequence from \cite{TomFanxin}. 
    We show that the existence of sequence of a asymptotic growth $\pi$ which infinitely guesses a probability one set is equivalent to the divergence of $\sum_{n=0}^{\infty}\frac{\pi(n)}{2^n}$. We then provide concrete examples for guessing sequences of certain low asymptotic growth using random walks. 
    Finally, we show that the ultrafilter construction from \cite{TomFanxin} always yield an ultrafilters and a sequence which guesses a meager set, while a simple construction using Cohen forcing gives a non-meager set of guessed reals. These results answer \cite[Question 6.13]{TomFanxin} and partially addresses \cite[Question 6.8]{TomFanxin}.
\end{abstract}

\maketitle

\section{Introduction}

Combinatorial properties of ultrafilters have a wide range of applications, from graph colorings to compact Hausdorff spaces to voting systems. 
A particular aspect of interest, which gained renewed interest in recent years, is the study of the cofinal structure of an ultrafilter, or more specifically, the structure of the Tukey order restricted to ultrafilters.

In \cite{Tukey40}, Tukey introduced the Tukey ordering on directed sets in order to analyze Moore-Smith convergence in topology.
We define a \textit{partial order} $\leq_A$ on a set $A$ as a relation that is reflexive, (weakly) antisymmetric, and transitive. 
A partial order $( A,\leq_A)$ is \textit{directed} if for any distinct $a_0,a_1\in A$, there exists $a_3\in A$ such that $a_1\leq_A a_3$ and $a_2\leq_A a_3$. 
Our primary interest is in the reverse inclusion order $\supseteq$ on an ultrafilter.
We say a subset $X\subseteq A$ is \textit{cofinal} if for every $a\in A$, there exists $x\in X$ such that $a\leq x$. 
A function $f:A\to B$ is \textit{cofinal} if the image of every cofinal subset in $A$ is cofinal in $B$. 
We say that $(B,\leq_B)$ is \textit{Tukey reducible} to $(A,\leq_A)$, denoted $B\leq_T A$, if there exists a cofinal map $f:A\to B$. 
If $A\leq_T B$ and $B\leq_T A$, we say that $A$ is Tukey equivalent to $B$ and denote it $A \equiv_T B$. 
We call $\equiv_T$-equivalence classes \textit{Tukey types}. 

We consider the Tukey types of directed partial orders of size continuum, a class of types which admits a maximal Tukey type: $( [\mathfrak{c}]^{<\omega},\subseteq)$. 
Moreover, Tukey \cite{Tukey40} provided a combinatorial characterization for being Tukey top for objects of size continuum. 

\begin{theorem}
    $(A,\leq_A)\equiv_T ([\mathfrak{c}]^{<\omega},\subseteq)$ if and only if there exists a subset $X\subseteq A$ of size continuum such that every $B\in [A]^\omega$ is unbounded in $(A,\leq_A)$. 
\end{theorem}

As we already mentioned, in this paper we are primarily interested in the Tukey order of ultrafilters on $\omega$, which are indeed objects of size continuum - a line of research that has extensively studied in e.g. \cite{SatInCan,TomNatasha2,TomNatasha,TomGabe,Blass/Dobrinen/Raghavan15,Dobrinen/Todorcevic11,Raghavan/Todorcevic12}. 
This study was initiated by Isbell \cite{Isbell65}, who showed that there exists an ultrafilter that is Tukey equivalent to $([\mathfrak{c}]^{<\omega},\subseteq)$. 
Ultrafilters that realize the maximal Tukey type are called \textit{Tukey-top}. 
A natural question that arises from Isbell's work is whether there are non-Tukey-top ultrafilters, or equivalently, are all ultrafilters Tukey-equivalent? 
This became known as Isbell's problem. 

As of recently, the full independence of Isbell's problem has been established. In \cite{Milovich08}, Milovich showed that p-points are below the maximum Tukey type. 
Moreover, Cancino and Zapletal showed its consistent that every nonprincipal ultrafilter is Tukey-top in \cite{cancinomanríquez2025isbellproblem}.
For more on Tukey types of ultrafilters see \cite{DobrinenTukeySurvey15}. 

In \cite{TomFanxin}, the first author and Wu provided an alternative construction to Tukey top ultrafilters using a guessing principle denoted by $\diamondsuit^-(\mathcal{U})$, where $\mathcal{U}$ is an ultrafilter on $\omega$  (see Definition \ref{def:1}). 
This definition is based on earlier work of the first author and Goldberg \cite{TomGabe} on the Tukey order at the level of measurable cardinals. 
More specifically, the principle ensures the existence of a sequence $\mathcal{A}_n\subseteq P(n)$ such that:
    \begin{enumerate}
        \item There are $\mathfrak{c}$-many reals $x$ such that $B_x:=\{n<\omega\mid x\cap n\in\mathcal{A}_n\}\in\mathcal{U}$.
        \item $n\mapsto|\mathcal{A}_n|$ has slow asymptotic growth.
    \end{enumerate}
Condition $(2)$ will be made precise in \ref{def:1}.    
The authors then prove that if an ultrafilter $\mathcal{U}$ on $\omega$ satisfies $\diamondsuit^-$, then $\mathcal{U}$ is Tukey-top. 
This is what motivates the further study of diamond sequences. 

Furthermore, using the Borel-Cantelli lemma, they show that the set of reals which are guessed by a sequence as in $\diamondsuit^-(\mathcal{U})$ cannot include all reals, and in particular the principle from \cite{TomGabe} cannot hold for ultrafilters on $\omega$. Concretely, they show that if $\A=\langle \mathcal{A}_n\mid n<\omega\r$ is a sequence of sets such that $\sum \frac{|\mathcal{A}_n|}{2^{f(n)}}$ converges, then  the set
$$\text{G}(\A):=\{x\in P(\omega)\mid B_x\text{ is infinite}\}$$
called the "set of guessed reals", must be probability zero. 
In particular, for any ultrafilter $\mathcal{U}$ and sequence of $\mathcal{A}_n$'s satisfying condition $(1)-(2)$,
$$\text{G}_{\mathcal{U}}(\A):=\{x\in P(\omega)\mid B_x\in\mathcal{U}\}$$
must have probability $0$. Then the authors asked whether the set $\text{G}_{\mathcal{U}}(\A)$ must always be small in other senses as well. Since measure and category are the most significant notions of smallness for sets of reals (i.e. the Lebesgue null ideal and the ideal of meager sets) the following question is of interest:
\begin{question}[{\cite[6.8]{TomFanxin}}]\label{question1}
    What sort of ”size” restriction on the set of reals which are guessed by an ultrafilter is consistent? Alternatively, what sort of topological properties are consistent to hold on a set of reals which are guessed by an ultrafilter?
\end{question}
Our first result is a positive answer to that question, i.e. $\text{G}_{\mathcal{U}}(\A)$ is meager, for the type of ultrafilters $\mathcal{U}$ constructed in \cite{TomFanxin}. Then we provide a simple example of a model (i.e. adding one Cohen real) where there is an ultrafilter $\mathcal{U}$ such that $\text{G}_{\mathcal{U}}(\A)$ is non-meager.

Another natural question from \cite{TomFanxin} is whether the asymptotic lower bound that $\sum\frac{\pi(n)}{2^{f(n)}}=\infty$ is tight for there to be a guessing sequence $\A$ such that $\text{G}(\A)$ has positive probability. 
\begin{question}
    How asymptotic-small can a function $\pi:\mathbb{N}\to\mathbb{N}$ be so that there is a seqeunce $\langle \mathcal{A}_n\mid n\in\mathbb{N}\rangle$ such that $|\mathcal{A}_n|\leq\pi(n)$, and the set $\text{G}(\A)$ has positive (standard borel) probability?  
\end{question}

In this paper we answer this question and show that $\sum_{n=0}^{\infty}\frac{\pi(n)}{2^n}$ is a threshold, namely if $\sum_{n=0}^{\infty}\frac{\pi(n)}{2^n}=\infty$, then there is a seqeunce (and in fact many sequences) $\A=\l \mathcal{A}_n\mid n<\omega\r$ such that $\text{G}(\A)$ has probability $1$. 
While this theorem is an existence theorem, the problem becomes how to find specific sequences:
\begin{question}
    Given a function $\pi:\mathbb{N}\to\mathbb{N}$ such that $\sum_{n=0}^{\infty}\frac{\pi(n)}{2^n}=\infty$. Construct natural examples of $\mathcal{A}_n$'s such that $\text{G}(\A)$ has probability one.
\end{question}
In section \ref{Sec: concrete}, we first construct some simple examples. 
Then we show that random walks give rise to natural examples for the functions $\pi(n)=\binom{n}{\frac{n}{2}}$ and $\binom{\frac{n}{2}}{\frac{n}{4}}^2$.
To provide these examples, we will use properties of random walks on lattices and P\'olya's theorem (see theorem \ref{polya}). 

\section{Guessing Principles of Ultrafilters on $\omega$.}

Due to the dual nature of this paper, we will provide both set theoretic and probabilistic detailed preliminaries. 
Let us start with introducing the notion of a filter, which are intended to provide an abstract notion of largeness.
A filter $\mathcal{F}$ on a set $A$ is a subset of the powerset of $X$, $P(X)$, such that: \begin{enumerate}
    \item $A\in \mathcal{F}$, $\emptyset\notin \mathcal{F}$
    \item If $A\in \mathcal{F}$ and $A\subseteq B$, then $B\in \mathcal{F}$
    \item If $A,B\in\mathcal{F}$, then $A\cap B\in\mathcal{F}$
\end{enumerate}
Additionally, we say that a filter $\mathcal{F}$ is an ultrafilter if 
\begin{enumerate}
    \item[(4)] For every $A\subseteq X$, either $A\in \mathcal{F}$ or $X\setminus A\in\mathcal{F}$.
\end{enumerate}
Every filter induces a dual notion of \textit{smallness} called the \textit{dual ideal}. 
In more formal terms, the dual ideal to a filter $\mathcal{F}$ on the set $X$ is the set $\mathcal{F}^*=\{A\subseteq X\:|\: X\setminus A\in \mathcal{F}\}$. 
Notice that if $\mathcal{F}$ is a filter that is not maximal, then the filter and its dual ideal do not cover all of the powerset of $X$. 
For a filter $\mathcal{F}$ and dual ideal $\mathcal{F}^*$, the set of positive sets is $\mathcal{F}^+=\{A\subseteq X\:|\:A\notin \mathcal{F}^*\}$. 
A filter is \textit{nonprincipal} if it does not contain any singletons.
If $\mathcal{U}$ is an ultrafilter, then it is nonprincipal if and only if it contains the Fr\'echet filter. 

Let the Fr\'echet filter $\mathcal{FR}$ on $X$ be a filter whose dual ideal is $[\omega]^{<\omega}$-- the collection of finite subsets of $,\omega$. 
More generally, let $\lambda $ be cardinal and $A$ be any set, denote by:
\begin{itemize}
    \item $[A]^\lambda=\{X\subseteq A\:|\: |X|=\lambda\}$.
    \item $[A]^{\leq\lambda}=\{X\subseteq A\:|\: |X|\leq\lambda\}$.
    \item $[A]^{<\lambda}=\{X\subseteq A\:|\: |X|<\lambda\}$.
\end{itemize}

Let $\gamma$ be a limit ordinal, we call a set $A\subseteq \gamma$ \textit{unbounded} in $\gamma$ if $\sup(A)=\gamma$.
A set $A\subseteq \gamma$ is \textit{closed} in $\gamma$ if for all limit $\alpha<\gamma$, if $\sup(A\cap\alpha)=\alpha$, then $\alpha\in A$. $A$ is a \textit{club} in $\alpha$ if it is closed and unbounded in $\alpha$.  For an ordinal $\kappa$ of uncountable cofinality we 
define $$\text{Cub}_\kappa=\{ X\subseteq \kappa \:|\: \exists C\text{ a club in }\kappa, C\subseteq X\}.$$ 
The set $\text{Cub}_\kappa$ is a filter over $\kappa$ which is called the \textit{club filter on $\kappa$}.
A set $S\subseteq \kappa$ is \textit{stationary} if it is $\text{Cub}_\kappa$-positive   i.e. for every club $C\subseteq \kappa$, $S\cap C\neq \emptyset$.

Recall the definition of Jensen's $\diamondsuit$ principle from \cite{Jensen1972}:  there exists a sequence $\l A_\alpha\:|\: \alpha<\omega_1\r$ such that $A_\alpha\subseteq \alpha$ and for any subset $X\subseteq \omega_1$, the set $\{\alpha<\omega_1\:|\: X\cap\alpha =A_\alpha\}$ is stationary. 
Unfortunately, it is impossible to generalize $\diamondsuit$ to countable sequences. 
Howbeit, a weakening of $\diamondsuit$, still due to Jensen, can be generalized in a useful way. Let us first recall that for a cardinal $\kappa$ of uncountable cofinal:
\begin{itemize}
    \item $\diamondsuit^-(\kappa)$ holds if there is a sequence $\l \mathcal{A}_\alpha\:|\: \alpha<\kappa\r$ such that $\mathcal{A}_\alpha\subseteq P(\alpha)$ and $|\mathcal{A}_\alpha|\leq |\alpha|$. For every $X\subseteq \kappa$, the set $\{\alpha \in \kappa \:|\:X\cap \alpha\in\mathcal{A}_\alpha\}$ is stationary in $\kappa$. 
    \item $\diamondsuit^*(\kappa)$ holds if for every $X\subseteq \kappa$, the set $\{\alpha\in\kappa\:|\: X\cap \alpha\in\mathcal{A}_\alpha\}$ is moreover a club in $\kappa$.
\end{itemize}
These principles, as they are, are obviously meaningless on $\omega$ as the club filter is not defined there. Nonetheless, on $\omega$ replacing the club filter with a filters is a possible:
\begin{definition}
    Let $\mathcal{F}$ be a filter over a cardinal $\kappa \geq \omega$, $\pi: \kappa \to \kappa$ be a function, and $T \subseteq P(\kappa)$. 
    Then, $\diamondsuit^-(\mathcal{F},\pi,T)$ holds if there is a sequence $\l \mathcal{A}_\alpha \:|\:\alpha < \kappa \r$ such that:
    \begin{enumerate}
        \item $\mathcal{A}_\alpha \subseteq P(\alpha)$,
        \item $|\mathcal{A}_\alpha |\leq \pi(\alpha)$,
        \item $\forall X \in T, \{\alpha < \kappa \: | \: X \cap \alpha\in \mathcal{A}_\alpha\}\in \mathcal{F}^+$,
    \end{enumerate}
    We say that $\diamondsuit^*(\mathcal{F},\pi,T)$ holds when there is a sequence $\l\mathcal{A}_\alpha\:|\:\alpha<\kappa\r$ that satisfies (1), (2), and 

    \begin{enumerate}
        \item[(3*)]  $\forall X \in T, \{\alpha < \kappa\:|\:X\cap \alpha\in \mathcal{A}_\alpha\}\in \mathcal{F}$.
    \end{enumerate}

\end{definition}   

\begin{remark}
    The principle $\diamondsuit^-(\text{Cub}_\kappa,id,P(\kappa))$ is the same as Jensen's $\diamondsuit^-(\kappa)$, and $\diamondsuit^*(\text{Cub}_\kappa,id,P(\kappa))$ is $\diamondsuit^*$. 
    Also, if $\mathcal{U}$ is an ultrafilter then $\diamondsuit^-(\mathcal{F},\pi,T)$ coincides with $\diamondsuit^*(\mathcal{F},\pi,T)$.
\end{remark}

Since we do not restrict the parameters, there are trivial instances of $\diamondsuit^-(\mathcal{F},\pi,T)$ e.g. if $\pi(\alpha)=2^{\alpha}$, one can take $\mathcal{A}_\alpha=P(\alpha)$. The strength of these principles is obtained by restricting $\pi$ to be as asymptotically small as possible and making $T$ as large as possible.

To formulate the restriction we have on $\pi$, let us introduce the relation on functions of almost equivalent and almost dominatio.  
Let $\mathcal{F}$ be a filter on a set $X$, and $f,g:X\to \kappa$.
Then, we define \begin{itemize}
    \item $f=_\mathcal{F}g$ if $\{x\in X\:|\: f(x)=g(x)\}\in\mathcal{F}$.
    \item $f\leq_\mathcal{F} g$ if $\{x\in X\:|\:f(x)\leq g(x)\}\in\mathcal{F}$.
    \item $f<_\mathcal{F} g$ if $\{x\in X\:|\:f(x)< g(x)\}\in\mathcal{F}$.
\end{itemize} 
Notice that $\leq_{\mathcal{F}}$ is linear if and only if $\mathcal{F}$ is an ultrafilter. 
In \cite{TomFanxin}, to get non-trivial and fruitful restrictions on $\pi$, the authors used the terminology of skies and constellations given by Puritz \cite{PURITZ1972215}. 

\begin{definition} (Puritz)
    Let $\mathcal{F}$ be a filter on $\omega$, $f,g:\omega\to\omega$. Define $\sky([g]_\mathcal{F})<sky([f]_\mathcal{F})$ if $\forall h:\omega\to\omega$,  $h\circ g <_\mathcal{F} f$. 
    Moreover, define $\sky([g]_\mathcal{F})\geq \sky([f]_\mathcal{F})$ if $\exists h:\omega\to\omega$, $h\circ g\geq_\mathcal{F} f$. 
\end{definition}

Notice that these are well-defined transitive relations on $\omega^\omega/=_{\mathcal{F}}$. In particular, the preorder $\leq$ induces an equivalence relation. We call each equivalence class a \textit{sky}. There exists a largest equivalence class, $sky([id]_{\mathcal{F}})$, and a least equivalence class consisting of the constant functions. 
\begin{definition} \label{def:1}
    We say that $\diamondsuit^-(\mathcal{U})$ holds for an ultrafilter $\mathcal{U}$ over $\omega$ if there exists $\pi,f:\omega\to\omega$ and $T\subseteq P(\omega)$ such that 
    \begin{enumerate}
        \item $\diamondsuit^-(\mathcal{U},\pi,T)$
        \item $|T|=2^{\aleph_0}$
        \item $\sky([\pi]_\mathcal{U})<\sky([id]_\mathcal{U})$
    \end{enumerate}
\end{definition}
The importance of this principle is justified by the following theorem: 
\begin{theorem}[{\cite[1.14]{TomFanxin}}]
    $\diamondsuit^-(\mathcal{U})\Rightarrow\mathcal{U}$ is Tukey-top.
\end{theorem}
Given a guessing sequence $\A=\l \mathcal{A}_n\mid n<\omega\r$ we consider \textit{the set of guessed reals}:
$$\text{G}(\vec{A})=\big\{X\mid \text{for infinitely many }n, X\cap n\in \mathcal{A}_n\big\}$$
Given a filter $\mathcal{F}$, we also let
$$\text{G}_{\mathcal{F}}(\vec{A})=\Big\{X\mid \{n\mid X\cap n\in \mathcal{A}_n\}\in\mathcal{F}\Big\}$$
Clearly, if $\mathcal{U}$ is a nonprincipal filter on $\omega$, $\text{G}_{\mathcal{F}}(\vec{A})\subseteq \text{G}(\vec{A})$. 

In this section, let us focus for now on the \textit{topological} small, i.e. meager. Recall that a subset $D$ of a topological space $A$ is dense if the closure of $D$ is equal to $A$, denoted $\overline{D}=A$. 
Similarly we say that $D\subseteq A$ is dense in $B\subseteq A$ if $D\cap B$ is dense in $B$. 
The complement of the closure of a set is the interior, denoted $D^\circ$, which is the largest open set contained in $D$. 
We say that a subset $X\subseteq A$ is \textit{nowhere dense} when there does not exist an open set $U\subseteq A$, such that $X$ is dense in $U$. 
Equivalently, $X\subseteq A$ is nowhere dense if $(\overline{X})^\circ=\emptyset$. 
We call a set $X\subseteq A$ \textit{meager} if $A$ is the countable union of nowhere dense sets. 
It can be shown that meager subsets of a set $X$ form an ideal which is closed under countable unions. Define a \textit{comeager} subset $A\subseteq X$ as the complement of a meager set.
Recall that a set $A\subseteq X$ is $G_\delta$ if it is the countable  intersection of open sets. 
 \begin{fact}\label{fact}
      Every dense $G_\delta$ set is comeager.
 \end{fact}
 We now focus on a specific space - the \textit{Cantor space} $2^\omega$. Let $2^{<\omega}=\bigcup_{n<\omega}2^n$ where $2^n$ is the  collection of finite functions $f:n\to 2$. We will frequently identify between $2^n$ and $P(n)$, $2^{<\omega}$ and $P(\omega)$ with $2^\omega$.
   For two functions $f,g$, we say that $g$ is an \textit{initial segment} of $f$ or that $f$ \textit{extends} $g$,  and write $g\sqsubseteq f$, if $f\restriction \text{dom}(g)=g$. The topology on $2^\omega$ is generated by the cylinder sets $C_t=\{f\in 2^\omega\:|\: f \text{ extends } t\}$ for all $t\in 2^{<\omega}$. 

 \begin{fact}
     Let $\vec{A}=\l \mathcal{A}_n\mid n<\omega\r$ be a sequence such that $\mathcal{A}_n\subseteq P(n)$. Then $\text{G}(\vec{A})$ is a $G_\delta$-set in $2^\omega$.
 \end{fact}

The next simple example shows that it is possible for a set of guessed reals to be comeager. 

\begin{example}
    Enumerate $2^{<\omega}=\{t_0,t_1,t_2,\ldots\}$ and choose any $x_n\in C_{t_n}$ where $C_{t_n}=\{X\in 2^\omega\:|\: t_n\sqsubset X\}$.  
    Partition $\mathbb{N}$ into infinitely many infinite sets $Y_n$ for $n<\omega$. Define for each $m\in Y_n$, $\mathcal{A}_m=\{x_n\cap m\}$. So $x_n$ is guessed on $Y_n$ and therefore the dense set $\{x_n\mid n<\omega\}$ is included in $\text{G}(\vec{A})$. By the two previous facts, we conclude that $\text{G}(\vec{A})$ is comeager.
\end{example}

\begin{corollary}
    For some sequence $\vec{A}=\l \mathcal{A}_n\mid n<\omega\r$, with $|\mathcal{A}_n|=1$, $\text{G}(\vec{A})$ is comeager.
\end{corollary}
Let us address now the Question \ref{question1}. First, note that the original construction of a $\diamondsuit^-(\mathcal{U})$ from \cite{TomFanxin} does not work. To see this let us review the construction.
Let us establish some notations. Given $S\subseteq 2^{<\omega}$, we let $\mathcal{L}_n(S)=S\cap 2^n$. We say that $r\in 2^\omega$ is a \textit{branch} through $S$ if there are infinitely many $n$'s such that $r\restriction n\in \mathcal{L}_n(S)$, and denote the set of branches by
$$Br(S)=\{r\in 2^n\mid r \text{ is a branch through }S\}.$$
The construction of $\diamondsuit^-$ ultrafilter used a special $S$ whose properties are described in the following lemma for \cite{TomFanxin}:
\begin{lemma}\label{mainlemma}
    Let $\pi:\omega\to\omega$ be any infinite-to-one function. Then there are two sequences $\l S_n\mid n<\omega\r$ and $\l m_n\mid n<\omega\r$ such that:
\begin{enumerate}
        \item [(i)] $S_n\subseteq 2^{\leq m_n}$.
        \item [(ii)] if $n\leq m$, then $S_n=S_m\cap 2^{\leq m_n}$ (so $m_n$ is also increasing).
        \item [(iii)] for each $k<m_n$, $S_n$ has at most $\pi(k)$-many nodes at level $k$.
\item [(iv)]  for every $n$, any node in $t\in S_{n+1}\setminus S_n$ is above some maximal node of $S_n$ (a node in $S_n$ is maximal if it has no proper extension in $S_n$).

    \item [(v)] if $n$ is even then every maximal node of $S_n$ branches in $S_{n+1}$ (and has exactly 2 extensions in $S_{n+1})$.
\item [(vi)] if $n$ is odd then:
\begin{enumerate}
    \item none of the maximal nodes of $S_n$ branches in $S_{n+1}$ (so the part of $S_{n+1}$ above that node forms a $\sqsubseteq$-chain).
    \item for any finite subset $F$ of the set of maximal nodes in $S_n$, there is a level $m_n\leq k<m_{n+1}$ such that $|F|=\pi(k)$  and every node in $F$ has an extension in $S_{n+1}$ at level $k$. In particular $\mathcal{L}_k(S_{n+1})$ consists exactly of those extensions.
\end{enumerate}
\end{enumerate}
\end{lemma}
Once the $S_n$ are constructed we let $S=\bigcup_{n<\omega}S_n$ and let $\mathcal{F}$ be the filter generated by the sets $B_r:=\{n<\omega\mid r\restriction n=\mathcal{L}_n(S)\}$ for $r\in Br(S)$. 
Then \cite[Cor 1.20]{TomFanxin} shows that it is possible to extend $\mathcal{F}$ to an ultrafilter $\mathcal{U}$ satisfying $\diamondsuit^-(\mathcal{U})$ witnessed by $\vec{A}$, where $\mathcal{A}_n=\mathcal{L}_n(S)$. 
Since we made sure that every branch of $S$ is guesses on a set in the ultrafilter, we obtain that $\text{G}_{\mathcal{U}}(\vec{A})=\text{G}(\vec{A})=Br(S)$. 
Let us denote by $d_n$ to be the number of maximal nodes in $S_n$. 
Thus:
\begin{enumerate}
    \item [(I)] $d_n=2^{\lceil\frac{n}{2}\rceil}$ (follows from (v) and (vi)a)
    \item [(II)] $m_{n+1}-m_n\geq 2^{d_n}$. (follows from (vi)b)
\end{enumerate}
\begin{theorem}
     Let $S$ be a tree constructed from Lemma \ref{mainlemma}, then $Br(S)$ is a meager set.
\end{theorem}
The proof uses the characterization of meager sets in $2^\omega$ which can be found in \cite[Proposition 5.3]{Blass2010}. Let $M\subseteq 2^\omega$,
\begin{equation}
    M\text{ is comeager iff }\exists x\in 2^\omega\exists \Pi=\l k_n\mid n<\omega\r, \ Match(x,\Pi)\subseteq M.
\end{equation}
where $Match(x,\Pi)$ consists of all $y$ such that for infinitely many $n$'s $y\restriction [k_n,k_{n+1})=x\restriction [k_n,k_{n+1})$.
\begin{proof}[\textit{Proof of Theorem.}]
    Consider $k_n=m_{2n+1}$ and let us define $x$. By (vi)b every maximal node in $S_{2n+1}\subseteq 2^{\leq m_{2n+1}}$ extends linearly to the levels in $[m_{2n+1},m_{2n+2})$. Let $x^n_1,...,x^n_{d_{2n+1}}$ enumerate the maximal nodes in $S_{2n+1}$. 
    
        For each $i$, let $y^n_i$ be a maximal branch above $x_i$ in $S_{2n+2}\setminus S_{2n+1}$ which is unique by (vi)a. 

    Since $x_i$ extends to at least  $2^{d_{2n+1}-1}$-many levels in $S_{2n+2}$ (by (vi)b), it follows that $y^n_i\in \mathcal{L}_k(S)$ for some $k\geq m_{2n+1}+2^{d_{2n+1}-1}$.
    Since $2^{d_{2n+1}-1}> d_{2n+1}$ (indeed $d_{2n+1}> 2$), we can define $x\restriction[m_{2n+1},m_{2n+1}+2^{d_{2n+1}-1})$ different than $$\{y^n_1\restriction [m_{2n+1},m_{2n+1}+2^{d_{2n+1}-1}),...,y^n_{d_{2n+1}}\restriction [m_{2n+1},m_{2n+1}+2^{d_{2n+1}-1})\}$$
    This defines $x$ on the intervals $[m_{2n+1},m_{2n+1}+2^{d_{2n+1}-1})$ and we extend $x$ arbitrarily so that $x\in 2^\omega$. 
    
    We claim that $Match(x, \l k_n\mid n<\omega\r)$ is disjoint from $Br(S)=\text{G}_{\mathcal{U}}(\vec{A})$. Indeed, let $r\in Br(S)$ be any branch. Since $r\restriction m\in \mathcal{L}_m(S)$ for infinitely many $m$'s, by (iv), there is a maximal branch $y$ in $S_{2n+2}$ such that $y\sqsubseteq r$. Again by (iv), there is $1\leq i\leq d_{2n+1}$ such that $x_i\sqsubseteq y$. By (vi)a and by the choice of $y^n_i$, $y=y^n_i$ for some $i$. It follows that $$x\restriction [m_{2n+1},m_{2n+1}+2^{d_{2n+1}-1})\neq y^i_n\restriction [m_{2n+1},m_{2n+1}+2^{d_{2n+1}-1})=r\restriction [m_{2n+1},m_{2n+1}+2^{d_{2n+1}-1}).$$ It follows that $x$ is different than $r$ on all the intervals and therefore $r\notin  Match(x,\l k_n\mid n<\omega\r)$.  
\end{proof}
We have the following partial answer to the question:
\begin{proposition}\label{thm comeager}
   After adding one Cohen real, there is an ultrafilter $\mathcal{U}$ over $\mathbb{N}$ with a sequence $\vec{A}$ witnessing $\diamondsuit^-(\mathcal{U})$, such that $\text{G}_{\mathcal{U}}(\vec{A})$ is non-meager. 
\end{proposition}
\begin{proof}
    Let $\pi$ be an infintie-to-one function i.e. for every $n$, $\pi^{-1}[\{n\}]$ is infinite. Let us force with $\text{Add}(\omega,1)$. We can identify $\text{Add}(\omega,1)$ with the forcing poset $\mathbb{P}$ consisting of all functions $f:n\to P(\omega)$ such that $f(n)\in [P(n)]^{\leq \pi(n)}$, ordered by end extension. Clearly, any generic for $\mathbb{P}$ produces a sequence $\vec{A}=\l \mathcal{A}_n\mid n<\omega\r$ such that $\mathcal{A}_n\in [P(n)]^{\leq n}$. By \cite[Theorem 2.1]{TomFanxin}, there is an ultrafilter $\mathcal{U}$ in the extension such that $\vec{A}$ witnesses that every ground model real is guessed on a $\mathcal{U}$-large set. It is well known (see for example \cite{Blass2010}) that the set of ground model reals after adding a single Cohen function is non-meager.
    \end{proof}
    The following question remains open:
\begin{question}
    Can there be a nonprincipal ultrafilter $\mathcal{U}$ such that $\diamondsuit^-(\mathcal{U})$ as witnessed by $\vec{A}$ such that $\text{G}_{\mathcal{U}}(\vec{A})$ is comeager?
\end{question}

\section{Probability Preliminaries}
As in the previous section, we often identify the elements of the Cantor space $2^\omega$ with subsets of $\omega$ via the indicator function. The \textit{standard Borel probability space} is the triple $(2^\omega,\Omega_0,\mathbb{P}_0)$, where $\Omega_0$ is the $\sigma$-algebra generated by the cylinders $C_t$ for $t\in 2^{<\omega}$ i.e. the Borel sets of the cantor space. The probability function $\mathbb{P}_0$ is then defined by setting $\mathbb{P}_0(C_t)=\frac{1}{2^{|t|}}$, where $|t|$ is the length of $t$. 

A theorem that provides some insights into guessing sequences is the Borel-Cantelli lemma.
Let $\l X_n\:|\: n<\omega\r$ be a sequence of sets. 
Then the limsup event is the event that infinitely-many of the $X_n$ occured, i.e. $\lim\sup X_n:=\bigcap_{n<\omega}\bigcup_{m\geq n}X_m$. The following is a basic result in probability theory \cite{Borel,Cantelli}:

\begin{lemma} [The Borell Cantelli Lemma]
    Let $\l E_n\:|\: n<\omega\r$ be a sequence of events in some probability space $\l X,\Omega,\mathbb{P}\r$.
    \begin{enumerate}
        \item If $\sum_{n=0}^\infty \mathbb{P}(E_n)<\infty$ then $\mathbb{P}(\underset{n\to\infty}{\limsup} \:E_n)=0$.
        \item If the  events are pairwise independent, and $\sum_{n=0}^\infty \mathbb{P}(E_n)=\infty$, then $\mathbb{P}(\underset{n\to\infty}{\limsup} \:E_n)=1$.
    \end{enumerate} 
\end{lemma}
A fundamental limitation on the possible size of $\text{G}_{\mathcal{U}}(\vec{A})$ is given then by the Borel-Cantelli lemma. 
In \cite[1.5]{TomFanxin}, it was noticed that given functions $f,\pi:\omega\to \omega$ and $\mathcal{A}_n\in [P(f(n))]^{\leq \pi(n)}$, either $\sum_{n=0}^{\infty}\frac{\pi(n)}{2^{f(n)}}=\infty$ or $\mathbb{P}(\text{G}(\vec{A}))=0$. 
Notice that if $\mathcal{U}$ is an ultrafilter and  $\sky([\pi]_\mathcal{U})<\sky([f]_\mathcal{U})$ then there is $\pi'\in [\pi]_{\mathcal{U}}$ and $f'\in[f]_{\mathcal{U}}$ such that $\sum_{n=0}^{\infty} \frac{\pi'(n)}{2^{f'(n)}}<\infty$.
Thus, a set of guessed reals must be probability zero. 
However, using the second Borel-Cantelli lemma, we get that if the aforementioned sum diverges, then the guessed set is probability one. 
\begin{theorem}\label{Thm: Other direction}
    If $\sum_{n=0}^{\infty}\frac{\pi(n)}{2^n}=\infty$, then there is a sequence $\vec{A}=\l \mathcal{A}_n\mid n<\omega\r$, with $|\mathcal{A}_n|=\pi(n)$ such that $\mathbb{P}(\text{G}(\vec{A}))=1$.
\end{theorem}
\begin{proof}
    Consider the standard Borel probability space $(\prod_{n\in\mathbb{P}}[P(n)]^{\leq \pi(n)},\Omega_1,\mathbb{P}_1)$, where $\Omega_1$ is the $\sigma$-algebra of Borel sets in the product topology of the discrete (finite) spaces  $[P(n)]^{\leq \pi(n)}$. The probability function $\mathbb{P}_1$ is defined by by setting $\mathbb{P}_1(``\mathcal{A}_n=A_n")=\frac{1}{\binom{2^n}{\pi(n)}}$ where $\mathcal{A}_0,\mathcal{A}_n,...$ is the random sequence. 
    
    Fix any $X\in P(\omega)$, and let $G_X$ be the event ``$X\in \text{G}(\vec{A})$" (recall that this is simply the event ``$\exists^\infty n, \ X\cap n\in\mathcal{A}_n$").  Note that $$\mathbb{P}_1(X\cap n\in \mathcal{A}_n)=\frac{\binom{2^n-1}{\pi(n)-1}}{\binom{2^n}{\pi(n)}}= \frac{\pi(n)}{2^n}$$
    Also note that the events $X\cap n\in\mathcal{A}_n$ and $X\cap m\in \mathcal{A}_m$ are independent for $n\neq m$ and therefore by the second Borel-cantelli Lemma, $\mathbb{P}_1(G_X)=1$. Now consider the Borel probability space $(2^\omega,\Omega_0,\mathbb{P}_0)$ from before. Consider the product space where we draw $X\in 2^\omega$ and $\vec{A}\in \prod_{n\in\mathbb{N}}[P(n)]^{\leq \pi(n)}$ independently. That is, the space ($2^\omega\times (\prod_{n\in\mathbb{N}}[P(n)]^{\leq n}, \Omega_\times,\mathbb{P}_\times)$, where $\Omega_\times$ is the $\sigma$-algebra generated by the products $A\times B$, ranging over $A\in \Omega_0$ and $B\in \Omega_1$, and $\mathbb{P}_{\times}(A\times B)=\mathbb{P}_0(A)\cdot\mathbb{P}_1(B)$. Consider the event $E=\{(\vec{A},X)\mid X\in \text{G}(\vec{A})\}$. To see that $E\in \Omega_{\times}$, note that $E=\lim\sup E_n$, where $E_n=\{(\vec{A},X)\mid X\cap n\in A_n\}$ and therefore it suffices to prove that each $E_n\in\Omega_\times$. Indeed, each $E_n$ is the union of finitely many cylinder sets (ranging over all possible $\mathcal{A}_n\in[P(n)]^{\leq\pi(n)}$ and every $Y\in A_n$). By Fubini's theorem we get that $$\mathbb{P}_\times(E)=\int_{\vec{A}}\mathbb{P}_0(E^{\vec{A}})d\mathbb{P}_1=\int_{X}\mathbb{P}_1(E_X)d\mathbb{P}_0$$
    where $E^{\vec{A}}=\{X\mid (\vec{A},X)\in E\}=\text{G}(\vec{A})$ and $E_X=\{\vec{A}\mid (\vec{A},X)\in E\}=G_X$.
    By the previous computation, 
    $$\int_{X}\mathbb{P}_1(E_X)d\mathbb{P}_0=\int_X 1d\mathbb{P}_0=1.$$
    Hence $\int_{\vec{A}}\mathbb{P}_0(\text{G}(\vec{A}))d\mathbb{P}_1=\int_{\vec{A}}\mathbb{P}_0(E^{\vec{A}})d\mathbb{P}_1=1$. In particular, there must exist at least one sequence  $\vec{A}$ such that $\mathbb{P}_0(\text{G}(\vec{A}))=1$ as wanted.
\end{proof}
Note that the previous theorem in fact shows that the set of $\vec{A}$ such that $\mathbb{P}_0(\text{G}(\vec{A}))=1$ has $\mathbb{P}_1$-probability 1.  In the next section we will present two natural such constructions: the first is very general and catches all possible $\pi$'s, and another one which comes out naturally from \textit{random walks}. Recall that given any graph, or in our case any lattice, a random walk starting at some initial vertex is a sequence of adjacent vertices such that there is a uniform probability for the next step to be one of the neighbors. 
More formally, we want to describe a random walk on a $d$-dimensional integer lattice $\mathbb{Z}^d$. 
For $n<\omega$, let $X_n$ be a random variable from the standard basis element of $\mathbb{Z}^d$, i.e., $2d$-many vectors with component $\pm1$ at a single coordinate, and the rest being $0$. 
Namely, $\mathbb{P}(X_n[i]=1)=\mathbb{P}(X_n[i]=-1)=\frac{1}{2d}$. 
These $X_n$'s will represent the $n$th step in a walk. 
Then, we define a random walk with $n$ steps as $S_n=x+\sum_{i=0}^n X_i$, where $x\in\mathbb{Z}^d$ is the initial point. 

A classical question to ask about random walks is whether or not a random walk returns to the origin and how often. A random walk is called \textit{recurrent} when with probability one, the walk returns to the origin infinitely often, otherwise the walk is \textit{transient}. 
Random walks on integer lattices are well-studied. 
For the results of this paper, we refer to a theorem of P\'olya, shown in \cite{Polya}.

\begin{theorem}(P\'olya) \label{polya}
    Suppose $S_n$ is a random walk starting at the origin. Then, if $d=1,2$, then the walk is recurrent. Otherwise, if $d\geq 3$, then the walk is transient. 
\end{theorem}
 
\section{Random Walks and Explicit Guessing Sequences}\label{Sec: concrete}

In this section we provides specific construction of guessing sequences. To warm up, let us start with a few trivial examples. 
By \cite[1.3]{TomFanxin}, these results hold for any $g\leq_\mathcal{FR} id$, so we assume that $f$ is maximal, i.e. $f=id$. 

\begin{example} \label{ex1}
    Let $\pi(n)=2^{n-1}$, and consider $\l \mathcal{A}_n\:|\:n<\omega \r$, $\mathcal{A}_n = P(n)\setminus P(n-1)$. Then $|\mathcal{A}_n|\leq2^{n-1}$ and $\mathbb{P}_0(\text{G}(\vec{A}))=1.$
\end{example}
\begin{proof}
    Indeed, $\vec{A}$ guesses $[\omega ]^\omega$ and $\mathbb{P}_0([\omega]^\omega)=1$. 
\end{proof}

Let us move on to the next smallest value of $\pi$ that is interesting. 

\begin{example} \label{ex2}
    Let $\pi(n)=2^{n-2}$. Consider the sequence $\l \mathcal{A}_n \:|\: n<\omega \r$ where $$\mathcal{A}_n = \{X\in P(n)\mid n-1\in x \land n-2 \in x\}.$$ Then  $|\mathcal{A}_n|=2^{n-2}$ and $\mathbb{P}_0(\text{G}(\A))=1$.
\end{example}
\begin{proof} 
    Clearly, $\G=\{x\in P(\omega)\:|\: \exists^\infty n(\{n,n+1\}\subseteq x)\}$.
    We show that the compliment is probability zero. 

    Define $E_n = \{x\in P(\omega)\:|\: \forall m > n (m \notin x \lor m+1\notin x)\}$. 
    It is easy to see that $(\G)^c = \bigcup_n E_n$. Next, we prove that for every $n<\omega$, $\mathbb{P}_0(E_n)=0$. Let us prove that $\mathbb{P}_0(E_0)=0$ and the argument for other $E_n$'s is essentially the same.

    Define $E_0^N = \{x\in P(N)\:|\: \forall n \leq N (n\notin x \lor n+1\notin x\}$. Then, $\bigcap_{N\in \mathbb{N}} E_0^N = E_0$ and it suffices to show that $\lim_{N\to\infty}\mathbb{P}_0(E_0^N)$ tends towards 0. It is easy to see that $|E_0^N| = F_N$, where $F_N$ is the $N$-th Fibonacci number. We have that $F_N=\frac{1}{\sqrt{5}}(\frac{1+\sqrt{5}}{2})^N-\frac{1}{\sqrt{5}}(\frac{1-\sqrt{5}}{2})^N$ from which it easily follows that $\lim_{N\to\infty}\frac{F_N}{2^N}=0$. 
\end{proof}

More generally we can ''freeze" any amount of bits  and the same argument works. This is a simple reformulation of the well-known Infinite Monkey Theorem \cite{Borel1913,LeHasard}. Which says that if a monkey randomly hits a keyboard with finitely many characters infinitely many times, each character is hit with a uniform probability, then any given finite string will appear infinitely often with probability $1$. 

\begin{proposition} \label{prop}
    Let $k<\omega$ and consider $\pi(n)=2^{n-k}$. For any $t\in 2^k$, define the sequence $\vec{A}^t$ by letting $$\mathcal{A}^t_n=\{x\in P(\omega) \: | \: \forall 0\leq q\leq k-1,n-q\in x\text{ iff }t(q)=1\}.$$ Then, $|\mathcal{A}^t_n|=2^{n-k}$ and $\mathbb{P}_0(G(\vec{A}^t))=1$.
\end{proposition}
\begin{proof} This is a straightforward application of the Infinite Monkey Theorem where we identify $2^\omega$ with $P(\omega)$. Note that $\text{G}(\vec{A}^t)$  translates precisely to the event `` the string $t$ repeats infinitely many times". 
\end{proof}
Already from these examples we can derive examples for a more general collection of $\pi$'s. 

\begin{corollary}
    If $\rho \leq^* \pi$ and $\A$ witnesses $\mathbb{P}(\G)=1$ where $|\mathcal{A}_n|\leq\rho(n)$, then a finite modification of $\A$ witnesses for $\pi$. 
    In particular, if $\lim \frac{\pi(n)}{2^n}>0$, then one of the sequences $\A^t$ from the previous proposition already witnesses for $\pi$ that  $\mathbb{P}(\G)=1$. 
\end{corollary}
\begin{proof}
Let $\A$ be defined as above. Fix $N$ such that for every $n\geq N$, $\rho(n)\leq \pi(n)$, and set $\mathcal{A}'_n=\emptyset$ if $n<N$ and $\mathcal{A}'_n=\mathcal{A}_n$ otherwise. Then Clearly $|\mathcal{A}_n|\leq \pi(n)$ and $\text{G}(\vec{A}')=\G$.

Let $\lim_{n\to\infty} \frac{\pi_1(n)}{2^n}=L$. We can then find $k$ large enough so that $\frac{1}{2^k}<L$ and it follows that $\pi_0(n)=2^{n-k}$ satisfies that $\pi_0\leq^* \pi_1$.  
\end{proof}

\begin{example}\label{ExampleMain}
    Let $\pi:\mathbb{N}\to\mathbb{N}$ be any function such that $\pi(n)\leq 2^n$ for every $n$ and $\sum_{n=0}^{\infty}\frac{\pi(n)}{2^n}=\infty$. Let us construct a sequence of $\mathcal{I}_n$ with $|\mathcal{I}_n|=\pi(n)$ such that  $\text{G}(\A)= P(\mathbb{N})$. Consider the lexicographic order on $2^n$, and for each $t\in 2^n$, let $0\leq p_n(t)<2^n$ be the position of $t$ in the linear order. Note that \begin{equation}\label{equation2}
        p_{n+1}(t^\frown0)=2p_n(t)\text{ and }p_{n+1}(t^\frown 1)=2p_n(t)+1. \end{equation} We denote by $\bar{0}$ and $\bar{1}$ the constant sequences with values $0$ and $1$ respectively. We inductively define intervals $\mathcal{I}_n$ of $2^n$ in the lexicographic order, along with $t_n$ and $s_n$ in $2^n$. Let $s_0=\bar{0}$ and $t_0=\overline{\pi(0)}$ (i.e. $\bar{0}$ or $\bar{1}$). Given $s_n$ and $t_n$, set \begin{align*}
        s_{n+1}=t_n^\frown\l 0\r \ \ \ \ \ \ \ \ \ \ \ \ \ \ \ \ \ \ \ \ \ \ \ \ \ \ \ \ \ \ \ \ \ \ \ \ \ \ \ \ \ \ \ \ \ \ \   \\
        t_{n+1}=p^{-1}_{n+1}\Big(\big(p_{n+1}(s_{n+1})+\pi(n+1) \big)\text{ mod } 2^{n+1}\Big)
    \end{align*}
    and set
    \begin{align*}
        \mathcal{I}_{n+1}=[s_{n+1},t_{n+1})\text{mod }2^{n+1}.
    \end{align*}
    More concretely, if $p_{n+1}(s_{n+1})+\pi(n+1)<2^{n+1}$ then $$p_{n+1}(s_{n+1})+\pi(n+1)\text{ mod }2^{n+1}=p_{n+1}(s_{n+1})+\pi(n+1).$$ By definition, $t_{n+1}$ is  the $\pi(n+1)$-th element past $s_{n+1}$, and $\mathcal{I}_{n+1}=[s_{n+1},t_{n+1})$. 
    
    Otherwise,  $2^{n+1}\leq p_{n+1}(s_{n+1})+\pi(n+1)$, then $$p_{n+1}(s_{n+1})+\pi(n+1)\text{ mod }2^{n+1}=p_{n+1}(s_{n+1})+\pi(n+1)-2^{n+1}.$$ So we cover the rest of the interval $[s_{n+1},\bar{1}]$ and start a new ``round" with the remainder. Namely, $\mathcal{I}_{n+1}=[s_{n+1},\bar{1}]\cup [\bar{0},t_{n+1}),$ where $t_{n+1}$ is the $p_{n+1}(s_{n+1})+\pi(n+1)-2^{n+1}$ element of the lexicographic order of $2^{n+1}$. 
    
    Note that in either case, $|\mathcal{I}_{n+1}|=\pi(n+1)$. Indeed, in the second case, $$|\mathcal{I}_{n+1}|=[2^{n+1}-p_{n+1}(s_{n+1})]+[p_{n+1}(s_{n+1})+\pi(n+1)-2^{n+1}]=\pi(n+1).$$ 
    It remains to argue that for each $f\in 2^\omega$ there are infinitely-many $n$'s such that $f\restriction n\in \mathcal{I}_n$.
    First argue by induction, that  \begin{enumerate}
        
        \item [(i)]$p_n(t_n)=(2^n\sum_{i=0}^n\frac{\pi(i)}{2^i})\text{ mod }2^n$.
        \item [(ii)] Between each two rounds, every $f\in 2^\omega$ has been guessed once.
        \item [(iii)] The number of rounds after $n$ steps is $\lfloor \sum_{i=0}^n\frac{\pi(i)}{2^i}\rfloor $
        
    \end{enumerate}
    From (iii) and our assumption there are infinitely many rounds, and by (ii) we conclude that every $f\in 2^\omega$ is guessed infinitely many times which concludes the proof. The proof of all three items is by induction. Item (i) easily follows from Equation (\ref{equation2}).
    To see (ii), suppose that $n_k< n_{k+1}$ are steps where we started new consecutive rounds. Let $f\in 2^\omega$. If $f\restriction n_k<t_{n_k}$, then $f\restriction n_k\in [\bar{0},t_{n_k})$, which is covered by $\mathcal{I}_{n_k}$. Otherwise, assume that $n_k\leq n<n_{k+1}$ is maximal such that $t_n\leq f\restriction n$. Then $s_{n+1}\leq f\restriction n+1$. If $n+1<n_{k+1}$, then by maximality $f\restriction n+1<t_{n+1}$ and $f\in\mathcal{I}_n$. If $n+1=n_{k+1}$, $f\in [s_{n_k},\bar{1}]\subseteq \mathcal{I}_{n_{k+1}}$. Hence each $f$ is guessed at some step $n_k\leq n\leq n_{k+1}$. Finally, to see (iii), we prove again by induction. At the induction step, we want to show that a new round was added i.e. $p_{n+1}(s_{n+1})+\pi(n+1)\geq 2^{n+1}$ if and only if $\lfloor \sum_{i=0}^{n+1}\frac{\pi(i)}{2^i}\rfloor =\lfloor \sum_{i=0}^n\frac{\pi(i)}{2^i}\rfloor +1$. To see the equivalence simply note that the former is equivalent to $2(2^n( \sum_{i=0}^{n}\frac{\pi(i)}{2^i} -\lfloor \sum_{i=0}^n\frac{\pi(i)}{2^i}\rfloor ))+\pi(n+1)\geq2^{n+1}$ and the latter is equivalent to $\frac{\pi(n+1)}{2^{n+1}}<1-(\sum_{i=0}^{n}\frac{\pi(i)}{2^i} -\lfloor \sum_{i=0}^n\frac{\pi(i)}{2^i}\rfloor )$. This concludes the proof.
\end{example}
We now start shrinking $\pi$ even further and deal with the class of $\pi$ such that $\lim_{n\to\infty}\frac{\pi(n)}{2^n}=0$ (but still $\sum_{n=0}^{\infty}\frac{\pi(n)}{2^n}=\infty$). 
Consider $\pi(n)=\binom{n}{\frac{n}{2}}$. Notice that this function satisfies the conditions we're looking for.

\begin{fact}

    $\binom{n}{\frac{n}{2}} \sim \frac{2^n}{\sqrt{n}}$. In particular,  $\lim_{n\to\infty} \frac{\binom{n}{\frac{n}{2}}}{2^n} = 0$ and $\sum_{n=0}^{\infty} \frac{\binom{n}{\frac{n}{2}}}{2^n}=\infty$.
    
\end{fact}
\begin{proof}
    Using Sterling's formula ($n! \sim\sqrt{2\pi n}(\frac{n}{e})^n$), we get $$\binom{n}{\frac{n}{2}}=\frac{n!}{(\frac{n}{2}!)^2} \sim \frac{\sqrt{2\pi n}(\frac{n}{e})^n}{(\sqrt{2\pi \frac{n}{2}}(\frac{\frac{n}{2}}{e})^\frac{n}{2})^2}=\frac{\sqrt{2}\cdot 2^n}{\sqrt{\pi n }}.$$ 
\end{proof}

A natural set of size $\binom{n}{\frac{n}{2}}$ is the set $\mathcal{A}_n=\{X\in P(n)\:|\:|x|=\frac{n}{2}\}$. We remark that unlike the previous examples, $\pi$ is a partial function defined on a non-cofinite set. Therefore, the ultrafilter must account for this.-- We can define $|x|=\lceil\frac{n}{2}\rceil$ and then it is defined everywhere. The function $\binom{n}{\lceil \frac{n}{2}\rceil}$ is still good for us.

To compute the probability of $\text{G}(\A)$ in this situation, it will be convenient to translate the problem to random walks on $\mathbb{Z}^d$. 
Random walks and guessing reals corresponds as follows: given a set $X\in P(\omega)$, the random walk corresponding to $X$ is given by a sequence $\l s_n\:|\: n<\omega\r$ such that if $n\in X$, then $s_n=1$, otherwise $s_n=-1$. Therefore, $X\cap n$ corresponds to $\sum_{i=0}^n s_i$, and $X$ is guessed by $\mathcal{A}_n$ when the walk returns to the origin. Since this is a one dimensional random walk on $\mathbb{Z}$, the walk visits the origin infinitely often. This is summarized in the following corollary

\begin{corollary}
    The sequence $\mathcal{A}=\{X\in P(n)\:|\:|X|=\frac{n}{2}\}$ satisfies that $\pi(n)=\binom{n}{\frac{n}{2}}$ and $\mathbb{P}(\text{G}(\A))=1$.
\end{corollary}

Notice, $\pi(n) = \binom{n}{\frac{n}{2}}$ relies on a one dimensional random walk on $\mathbb{Z}$. Therefore, we can shrink $\pi$ so that the random walk occurs on $\mathbb{Z}^2$. This gives us $\pi(n)=\binom{\frac{n}{2}}{\frac{n}{4}}^2$ which is asymptotically equivalent to $\frac{2^n}{n}$ .  

As with the earlier proposition, we would like to reduce guessing sets to a simple random walk on $\mathbb{Z}^2$. Partition $\omega$ into evens and odds, denoted  by $\mathbb{N}_{e}$ and $\mathbb{N}_{o}$ respectively. Then, for every $X\in P(\omega)$, there is a corresponding 2D random walk: First we translate each $X$ split $X$ into $X\cap \mathbb{N}_{e}$ and $X\cap\mathbb{N}_{o}$, which is turn codes two independent random walks as in the previous case. It is well known that the pair of two 1D random walks is a 2D random walk. Note that under this translation, the requirement that requiring that the 2D random walk return to the origin after $n$ many steps is the same as requiring that $|X\cap \mathbb{N}_e\cap n|=\frac{n}{4}=|X\cap \mathbb{N}_o\cap n|$. Hence, for this case we also have a natural example of guessing sets which correspond to the random walk. Namely, consider $\l\mathcal{A}_n\:|\:n<\omega\r$ where $$\mathcal{A}_n=\{X\in P(n)\:|\:|X\cap \mathbb{N}_{e}|=\frac{n}{4}= |X \cap \mathbb{N}_{o}|\}.$$
The size of $\mathcal{A}_n$  is  $\pi(n)=\binom{\frac{n}{2}}{\frac{n}{4}}^2$, and the following proposition shows that $\mathbb{P}(\text{G}(\A))=1$
\begin{proposition}
    Let $\A$ be the sequence defined by $\mathcal{A}_n=\{X\in P(n)\:|\:|X\cap \mathbb{N}_{e}|=\frac{n}{4}= |X \cap \mathbb{N}_{o}|\}.$ Then $|\mathcal{A}_n|=\binom{\frac{n}{2}}{\frac{n}{4}}^2$ and $\mathbb{P}(\text{G}(\A))=1$.
\end{proposition}
\begin{proof}
    
    Then, we can represent an element of $X\in T$ as a random walk in two dimensions. 
    Using the method previous described, let $\l s^X_n\:|\: n\in\omega\r$ be the corresponding random walk to $X$. By P\'olya's theorem, with probability one, every random walk returns to the origin infinitely often. Let $T=\{X\in P(\omega)\:|\:|\{n<\omega\:|\:X\cap n\in \mathcal{A}_n\}|=\aleph_0\}$. This is a probability one set. 
\end{proof}

\subsection{Acknowledgments} We are extremely thankful to G. Amir for his feedback on this paper and especially for his help proving Theorem \ref{Thm: Other direction} and Example \ref{ExampleMain}. We would also like to thank Fanxin Wu for fruitful discussions.
\bibliographystyle{amsplain}
\bibliography{ref}
\end{document}